\documentclass{amsart}
\usepackage[left=2cm, right=2cm]{geometry}
\usepackage[utf8]{inputenc}
\usepackage{amsfonts, amsthm, amssymb, mathtools,etoolbox,mathrsfs}
\usepackage{blindtext}
\usepackage[colorlinks=true, urlcolor=blue, linkcolor=blue, citecolor=blue]{hyperref}
\usepackage{tikz,tikz-cd}
\usetikzlibrary{arrows.meta} 
\usepackage{array}
\usepackage{comment}
\tikzset{pullback/.style={minimum size=1.2ex,path picture={
\draw[opacity=1,black,-,#1] (-0.5ex,-0.5ex) -- (0.5ex,-0.5ex) -- (0.5ex,0.5ex);%
}}}
\usepackage{cleveref}

\theoremstyle{plain}
\newtheorem{theorem}{Theorem}[section]
\newtheorem{proposition}[theorem]{Proposition}
\newtheorem{lemma}[theorem]{Lemma}
\newtheorem{corollary}[theorem]{Corollary}

\newtheorem{fact}[theorem]{Fact}

\theoremstyle{definition}
\newtheorem{example}[theorem]{Example}
\newtheorem{definition}[theorem]{Definition}
\newtheorem{remark}[theorem]{Remark}
\newtheorem{notation}[theorem]{Notation}

\newtheorem{setting}[theorem]{Setting}

\newcommand{\dq}[1]{``#1"}
\newcommand{\memo}[1]{}
\newcommand{\invmemo}[1]{}
\newcommand{\para}[1]{\paragraph{\textbf{#1}}}

\newcommand{\N}{\mathbb{N}}

\newcommand{\Z}{\mathbb{Z}}

\newcommand{\A}{\mathcal{A}}
\newcommand{\B}{\mathcal{B}}
\newcommand{\C}{\mathcal{C}}
\newcommand{\D}{\mathcal{D}}
\newcommand{\E}{\mathcal{E}}
\newcommand{\F}{\mathcal{F}}
\renewcommand{\L}{\mathcal{L}}
\newcommand{\id}{\mathrm{id}}
\newcommand{\ob}{\mathrm{ob}}

\newcommand{\op}{\mathrm{op}}
\newcommand{\Le}{\mathrm{Le}}

\newcommand{\Image}{\mathrm{Im}}
\newcommand{\Set}{\mathbf{Set}}
\newcommand{\FinSet}{\mathbf{FinSet}}
\newcommand{\PSh}{\mathbf{PSh}}

\newcommand{\OC}{\mathbb{A}}
\newcommand{\iOC}{\mathbb{A}^{\circ}}

\newcommand{\Func}[2]{[#1,#2]}
\newcommand{\Lstr}{\mathbf{Str}_{\L}}

\newcommand{\Geom}{\mathbf{Geom}}
\newcommand{\lar}[1]{\overset{#1}{\leftarrow}}
\newcommand{\demph}[1]{\textbf{\textit{#1}}}
\newcommand{\gmX}{\overline{X}}
\newcommand{\lm}{\lambda}
\newcommand{\Lm}{\Lambda}

\newcommand{\uniobj}{\mathscr{X}}
\newcommand{\uniinhobj}{\mathscr{X}^{\circ}}
\newcommand{\X}{\mathbb{X}}

\renewcommand{\P}{\mathscr{P}}

\newcommand{\neFinSet}{\mathbf{FinSet}^{\circ}}

\title[Solution to Lawvere's first problem]{Solution to Lawvere's first problem: \\a Grothendieck topos that has proper class many quotient topoi}
\author{Yuhi Kamio, Ryuya Hora}
\thanks{Graduate School of Mathematical Sciences, The University of Tokyo. \url{emirp13@g.ecc.u-tokyo.ac.jp}}
\thanks{Graduate School of Mathematical Sciences, University of Tokyo. \url{hora@ms.u-tokyo.ac.jp}}

\begin{document}


\begin{abstract}
    This paper solves the first of the open problems in topos theory posted by William Lawvere, concerning the existence of a Grothendieck topos that has proper class many quotient topoi. This paper concretely constructs such Grothendieck topoi, including the presheaf topos on the free monoid generated by countably infinitely many elements $\mathbf{PSh}(M_\omega)$. Utilizing the combinatorics of the classifying topos of the theory of inhabited objects and with the help of a system of pairing functions, the problem is reduced to a theorem of Vop\v{e}nka, Pultr, and Hedrl\'{ı}n, which states that any set admits a rigid relational structure.
\end{abstract}

\maketitle
\tableofcontents

\section{Introduction}
\para{Universal appearance and challenges of Quotient Topoi}
Quotient topoi of a topos are connected geometric morphisms from the topos (see \cref{SectionQuotientTopoi}). Quotient topoi, being very elementary concepts (in duality with subtoposes), have played a fundamental role in both the logical and geometrical aspects of topos theory. 
Examples of the appearance of quotient topoi (=connected geometric morphisms) include the theory of hyperconnected-localic factorization \cite{johnstone1981factorization}, local maps \cite{johnstone1989local}, isotropy group \cite{henry2018localic}, and cohesive topoi \cite{lawvere2007axiomatic}.
In terms of the analogy between logos theory (which is the dual algebra of the topos geometry) and ring theory as emphasized by Anel and Joyal \cite{anel2021topo}, a quotient topos (= a sublogos) corresponds to the concept of a subring in ring theory (See also \cite[6.3.6]{lurie2009higher}). 
In a manner of speaking, the universal appearance of quotient topoi in topos theory is analogous to the universal appearance of subrings in ring theory.


Contrary to the universal role of quotient topoi, comprehending the entirety of quotient topoi remains challenging. For instance, while the dual concept of subtopos is well-known to be classified by the Lawvere-Tierney topology, no such “internal parameterization” is known for quotient toposes. Although several papers \cite{rosenthal1982quotient,el2002simultaneously,henry2018localic, hora2024internal, hora2024quotient} have succeeded in classifying some quotient toposes of some toposes, even the number of quotient topoi is not known. More precisely, the question of whether the number of quotient topoi of a Grothendieck topos is small remains open. This is the open problem we will solve in this paper.

\para{Lawvere's open problems in topos theory}
William Lawvere has posed seven open problems in topos theory in \cite{Open240411LawvereRevised}. The problem of the number of quotient topoi is selected as the first of the seven open problems and called \dq{Quotient toposes}:
\begin{quote}
    [...] Is there a Grothendieck topos for which the number of these quotients is not small? [...] \cite{Open240411LawvereRevised}
\end{quote}
About the importance of these seven open problems, Lawvere says
\begin{quote}
    Clarification on them would further advance work on topos theory and its applications to thermomechanics, to algebraic geometry, and to logic. \cite{Open240411LawvereRevised}
\end{quote}


\para{Contribution}
This paper provides a complete solution to this open problem by concretely constructing a Grothendieck topos and its proper class many (mutually non-equivalent) quotient topoi (\cref{MainTheorem}).

In particular, \cref{CorollaryCountableMonoid} shows that for a countably infinitely generated free monoid $M_{\omega}$, the topos $\PSh(M_{\omega})$ has proper class many quotient topoi. This result could be regarded as a continuation of the authors' previous result \cite{hora2024quotient} of classification of quotient topoi for the topos of discrete dynamical systems $\PSh(\N)\simeq\PSh(M_1)$.

\para{Overview}
We will construct proper class many mutually non-isomorphic connected geometric morphisms $\E_{\L} \to \iOC$ from a well-designed Grothendieck topos $\E_{\L}$ to the topos of the classifying topos of the theory of inhabited objects $\iOC=\Func{\neFinSet}{\Set}$. Thanks to the universality of the classifying topos $\iOC$, the problem is reduced to the construction of inhabited objects of the topos $\E_{\L}$ (corresponding to connected geometric morphisms). \Cref{SectionLengthCombinatoricsAndClassifiers} is devoted to the combinatorics of the classifying topoi $\iOC$ and its embedding $\iota\colon \iOC \hookrightarrow \OC$ into the classifying topos of the theory of objects $\OC=\Func{\FinSet}{\Set}$. 

An inhabited object $X\in \ob(\E)$ should be \dq{sufficiently rigid}, which we will call \demph{inhabited-topos-rigid}, in order to make the corresponding geometric morphism $\E \to \iOC$ connected. Recall that an object $X$ of a category is \demph{rigid} if it has no non-trivial endomorphisms. Similarly, we will define stronger notions of rigidity in \cref{SectionRigidities} and gradually increase the rigidity in \cref{SectionLexRigidityAndLStructures,SectionPresheafEncodingAndMainTheorem} to construct inhabited-topos-rigid objects in the topos $\E_{\L}$.
\begin{description}
    \item[rigid] \Cref{PropositionManyRigid} in the category of $\L$-structures $\Lstr$
    \item[inhabited-lex-rigid] \Cref{PropositionManyNonEmptyLexRigid} in the category of $\L$-structures $\Lstr$
    \item[inhabited-topos-rigid] \Cref{MainTheorem} in the \dq{encoding topos(\cref{DefinitionPresheafEncoding})} $\E_{\L}$ 
\end{description}
To make lex-rigid objects and inhabited-topos-rigid objects, a pairing function plays a key role. In particular, to construct inhabited-topos-rigid objects, an $\aleph_0$-pairing function is used to control the image of a function (\cref{ProppositionImagefBounded}).



    
\subsection*{Clarification of the roles of the authors} The two authors have been discussing this problem since 2022. The initial rough proof was provided solely by the first-listed author, Yuhi Kamio. The second-listed author, Ryuya Hora, pointed out a singularity when the length is $0$ or $1$ and proposed a formulation using the classifying topos of the theory of inhabited objects. The writing of the paper was a collaborative effort.

\subsection*{Acknowledgement}
We would like to thank Ryu Hasegawa, who is the supervisor of the second-listed author, for his helpful discussions and suggestions.
We would like to thank Koshiro Ichikawa for some discussions about rigid structures and for answering technical questions about \cref{PropositionManyRigid}.
We would like to express deep gratitude to the anonymous reviewers who carefully read the manuscript and provided numerous valuable comments. This greatly assisted the inexperienced authors in enhancing the quality of the paper.

We are also grateful for the discussions at the \dq{mspace topos} and would like to extend our gratitude to its organizers, Toshihiko Nakazawa and Fumiharu Kato, for fostering such an encouraging environment.

The second-listed author is supported by FoPM, WINGS Program, the University of Tokyo.



\section{Preliminaries on Quotient topoi}\label{SectionQuotientTopoi}
This section aims to summarize the basic definitions and conventions of quotient topoi.
\begin{definition}
    A geometric morphism $f\colon \E \to \F$ is \demph{connected} if the left adjoint part $f^{\ast}$ is fully faithful.
\end{definition}
For the notion of connected geometric morphisms and its geometric intuition, see \cite{johnstone2002sketchesv2}.
\begin{definition}
    A \demph{quotient topos} of a topos $\E$ is a(n equivalence class of) connected geometric morphism from the topos $\E$.
\end{definition}

\begin{remark}[Equivalence of quotient topoi]
    Rigorously speaking, \dq{the number of quotient topoi} must mean the number of equivalence classes of quotient topoi in order to make it a meaningful question. Otherwise, every Grothendieck topos $\E$ has proper class many quotient topoi trivially since every equivalence of categories $\E \to \E'$ defines a quotient topos.
    
    Two quotient topoi $f\colon \E \to \F$ and $f' \colon \E \to \F'$ are \demph{equivalent} if there exists an equivalence of categories $e\colon \F \to \F'$ such that the following diagram is commutative up to natural isomorphism
    \[
    \begin{tikzcd}[row sep = 10pt]
        &\F\ar[dd,"e"{name=A}]\\
        \E\ar[ru,"f"]\ar[rd,"f'"']\ar[to = A,phantom, "\cong"]&\\
        &\F'.
    \end{tikzcd}
    \]
\end{remark}

\section{Length and minimal expression in the classifying topos of the theory of objects and inhabited objects}\label{SectionLengthCombinatoricsAndClassifiers}

\subsection{Length and minimal expression in the classifying topos of the theory of objects}

As a preparation for our construction of proper class many quotient topoi, we recall the notion of the classifying topos of the theory of objects, and define the notion of \demph{length} and \demph{minimal expression}.

\subsubsection{Definitions and basic properties}


\begin{notation}We adopt the following notations and terminologies.
\begin{itemize}
    \item We write $\OC$ for the (covariant) functor category $[\FinSet,\Set]$, which is known to be \demph{the classifying topos of the theory of objects}\footnote{This \dq{affine line} notation is borrowed from \cite{anel2021topo}.}. See \Cref{FactEquivObjectClassifier} for its universality.
    \item For two topoi $\E,\F$, the symbol $\Geom(\E,\F)$ denotes the category of geometric morphisms from $\E$ to $\F$ and natural transformations between the left adjoints.
    \item For non-negative integer $n$, $[n]$ denotes the $n$-element set $\{0,1,2,\dots ,n-1\}$. 
    \item The covariant functor $\FinSet \to \Set$ represented by the $n$-element set $[n]$ is denoted by $y[n]$. 
    \item The \demph{universal object} in $\OC$, which is the canonical embedding functor $\FinSet \to \Set$, is denoted by $\uniobj$. (The universal object $\uniobj$ is naturally isomorphic to the representable functor $y[1]$.)
\end{itemize}
\end{notation}

\begin{fact}[See \cite{maclane1994sheaves}]\label{FactEquivObjectClassifier}
    For a Grothendieck topos $\E$, the functor
    \[
    \Geom(\E,\OC)\to \E \colon f \mapsto f^{\ast}(\uniobj)
    \]
    is an equivalence of categories.
\end{fact}
Since this equivalence of categories is a fundamental tool throughout this paper, this subsection aims to recall the concrete correspondence and introduce related notations.

\begin{notation}\label{NotationTensorFunctor}
    For a Grothendieck topos $\E$ and a geometric morphism $f\colon \E \to \OC$ corresponding to an object $X \in \ob(\E)$, the left adjoint part $f^{\ast}\colon \OC \to \E$ will be denoted by $X\otimes {-}\colon \OC \to \E$. 
\end{notation}

\begin{remark}[As a left Kan extension]\label{remark:TensorAsKanext}
    The above tensor functor is nothing other than a left Kan extension along the Yoneda embedding. In fact, the functor $X\otimes {-}$ in \cref{NotationTensorFunctor} is the left Kan extension
    \[
    \begin{tikzcd}
        \FinSet^{\op}\ar[dr, hookrightarrow, "y"'] \ar[rr, "X^{\bullet}", ""'{name=A}]&{}&\E\\
        &\OC=\Set^{\FinSet}, \ar[ru, "\mathrm{Lan}_{y} X^{\bullet}"', ""{name=B}]\ar[from=A, to=2-2, Rightarrow, "\cong"',"\eta"]&
    \end{tikzcd}
    \]
    where $X^{\bullet}\colon \FinSet^{\op}\to \E$ is the lex functor that sends $[n]$ to $X^n$. This is a special case of a classifying topos of finite limit theory, which is explained in terms of Kan extension in some textbooks, including \cite[Theorem 4.2.1.]{borceux1994handbookv3}.
\end{remark}
By a general reason (for example, see \cite[X.4.]{Maclane1998CWMcategories}), this functor $X\otimes {-}$ is described in terms of coends as follows.
\begin{proposition}
    For a Grothendieck topos $\E$, an object $X \in \ob(\E)$, and an object $F\in \ob(\OC)$, the object $X\otimes F$ is given by the following coend:
    \[
    X\otimes F \cong  \int^{A\in \FinSet} X^A \times FA \textup{ (in $\E$)},
    \]
    where the object $X^A$ is the $A$-times product of the object $X$ and $X^A \times FA$ is its $FA$-times coproduct.
    In particular, if the topos $\E$ is the category of sets $\Set$ and the object $X$ is a set,
    an element of the set $X\otimes F$ is an equivalence class of $(X\lar{f} A, \sigma \in FA)$ with the equivalence relation generated by
    \[
    (X\lar{f\circ g}B, \sigma \in FB) \sim (X\lar{f} A, g\cdot  \sigma \in FA).
    \]
\end{proposition}


\begin{notation}[Tensor notation]\label{NotationTensorNotationOfElement}
    For a set $X \in \ob(\Set)$, 
    the equivalence class of $(X\lar{f} A, \sigma \in FA)$ will be denoted by $f\otimes \sigma \in X \otimes F$.
    This notation enables us to manipulate elements of $X\otimes F$ like
\[
(f\circ g) \otimes \sigma = f\otimes (g\cdot \sigma).
\]
In the case where $A=[n]$ for a non-negative integer $n$, the element $f\otimes \sigma$ is also denoted by $(f(0), \dots ,f(n-1)) \otimes \sigma$.
\end{notation}

\begin{notation}\label{NotationOtimesSigmaMorphism}
    For a Grothendieck topos $\E$, an object $X\in \ob(\E)$, an object $F\in \ob(\OC)$, a non-negative integer $n$, and an element $\sigma \in F[n]$, the corresponding morphism in the topos $\E$
    \[
    \begin{tikzcd}
        X^n \ar[r,"\cong"] & X\otimes y[n]\ar[r,"X \otimes \overline{\sigma} "] &X\otimes F,
    \end{tikzcd}
    \]
    is denoted by ${-}\otimes \sigma$. Here, the morphism $\overline{\sigma}\colon y[n]\to F$ is the morphism corresponding to the element $\sigma \in F[n]$ by the Yoneda lemma. 
\end{notation}

We will summarize some properties of the coend that we will use later.

\begin{lemma}\label{LemmaCoyoneda}
    In the topos of sets $\E= \Set$, if $X$ is a finite set and $(X\lar{\id_{X}}X)\otimes (\sigma\in FX)= (X\lar{\id_{X}}X)\otimes (\tau\in FX)$, then $\sigma=\tau$.
\end{lemma}
\begin{proof}
    This follows from the co-Yoneda lemma:
    \[
    FX \cong \int^{A\in \FinSet}\FinSet(A,X) \times FA.
    \]
\end{proof}


\begin{lemma}\label{LemmaNaturalityOfCoend}
    An object $F\in \OC$ defines a functor ${-}\otimes F \colon \E \to \E$. In the case where $\E=\Set$, an arbitrary function $g\colon X\to Y$ induces a well-defined function
    \[
    g\otimes F \colon X\otimes F \to Y\otimes F \colon f\otimes \sigma \mapsto (g\circ f) \otimes \sigma.
    \]
\end{lemma}
According to \cref{NotationTensorNotationOfElement}, if $A=[n]$, the above action of the function $g$ is also written as \[(x_0, \dots x_{n-1})\otimes \sigma \mapsto (g(x_0), \dots , g(x_{n-1}))\otimes \sigma.\]

\begin{proposition}\label{PropositionFunctorTensorDescription}
    For a small category $\C$, consider the covariant functor category $\Func{\C}{\Set}$. For a functor $\X \colon \C \to \Set$, and an object $F\in \OC$, the functor $\X \otimes F \colon \C \to \Set$ (in \cref{NotationTensorFunctor} where $\E=\Func{\C}{\Set}$) is given in the following way.
    \begin{itemize}
        \item For an object $c\in \ob(\C)$, $(\X\otimes F)(c) = \X(c) \otimes F$.
        \item For a morphism $f\colon c \to c' $ in the category $\C$, the action $(\X\otimes F)(c)\to (\X\otimes F)(c')$ is given by $\X(f)\otimes F \colon \X(c) \otimes F\to \X(c') \otimes F$.
    \end{itemize}
\end{proposition}
\begin{proof}
    This follows since products and colimits in $\Func{\C}{\Set}$ are pointwise.
\end{proof}

\subsubsection{Length and minimal expression}


Throughout this subsubsection, let $X$ denote a set and $F$ denote an object of $\OC$.

\begin{definition}[Length and minimal expression]
    For each element $t \in X\otimes F$,
    \begin{itemize}
         \item an \demph{expression} of $t$ is a representative $(X \lar{f}A,\sigma\in FA)$, such that $f\otimes \sigma =t$.
        \item the \demph{length} of an expression $(X \lar{f}A,\sigma\in FA)$ is the cardinality of $A$.
        \item the \demph{length} $\Le(t)$ of the element $t$ is the minimum length of expression among all expressions of $t$.
        \item the \demph{minimal expression} of $t$ is an expression, that realizes the (minimum) length of $t$.
    \end{itemize}
\end{definition}

\begin{lemma}\label{LemmaMinimalInjectivity}
    If $(X\lar{f} A, \sigma \in FA)$ is a minimal expression, then $f$ is injective.
\end{lemma}
\begin{proof}
    If $f$ is not injective, by replacing $f$ with its image inclusion, we obtain another expression with smaller length. That contradicts the minimality.
\end{proof}

\begin{lemma}\label{LemmaCancellationInjection}
    Suppose $(X\lar{f} A)\otimes( \sigma \in FA)=(X\lar{f} A)\otimes( \tau \in FA)$ , $f$ is injective and $A\neq \emptyset$. Then, $\sigma =\tau$. 
\end{lemma}
\begin{proof}
    Let $r\colon X\twoheadrightarrow A$ be a retraction of $f$. Applying $r\otimes F$ (\cref{LemmaNaturalityOfCoend}), we have 
    \[
    (A\lar{\id_{A}}A)\otimes (\sigma \in FA)=(A\lar{r}X\lar{f}A)\otimes (\sigma\in FA) = (A\lar{r}X\lar{f}A)\otimes (\tau\in FA)=(A\lar{\id_{A}}A)\otimes (\tau\in FA).
    \]
    Due to \cref{LemmaCoyoneda}, we obtain $\sigma = \tau$.
\end{proof}
\begin{lemma}[Comparison lemma, when $\Le(t)>1$]\label{LemmaWeakUniquenessMinimalExpression}
    Suppose $(X\lar{f} A, \sigma \in FA)$ is a minimal expression of an element $t\in X\otimes F$, and its length is greater than $1$. Then, for any other expression $(X\lar{g} B, \tau \in FB)$, we have $\Image(f) \subset \Image(g)$. 
\end{lemma}
\begin{proof}
    Suppose $\Image(f) \not \subset \Image(g)$ and take an element $a\in A$ such that $f(a) \notin \Image(g)$. Since $\#A > 1$, we can take another element $a' \in A\setminus\{a\}$. Let $h\colon X \to X$ be a function defined by 
    \[
    h(x)=
    \begin{cases}
        f(a') & (x=f(a))\\
        x & (\text{otherwise}).
    \end{cases}
    \]
    Then, the composite of $X\lar{h} X \lar{f} A$ is not injective, since  $h(f(a))=f(a') = h(f(a'))$.
    The composite of $X\lar{h} X \lar{g} B$ is equal to $g$. Applying $h\otimes F$ (\cref{LemmaNaturalityOfCoend}) the equation $f\otimes \sigma = g\otimes \tau$ implies the equation $(h\circ f)\otimes \sigma = (h\circ g)\otimes \tau$. Then, we have 
    \[
    (h\circ f)\otimes \sigma = (h\circ g)\otimes \tau = g\otimes \tau = f\otimes \sigma.
    \]
    Since the expression $(h\circ f, \sigma)$ has the same length as the minimal expression $(f,\sigma)$, it is also minimal. However,
    since $h\circ f$ is not injective,
    this contradicts 
    \cref{LemmaMinimalInjectivity}.
\end{proof}

\begin{proposition}[Uniqueness of minimal expression, when $\Le(t)>1$]\label{PropositionWeakUniquenessMinimalExpression}
    Suppose $(X\lar{f} A, \sigma \in FA)$ and $(X\lar{g} B, \tau \in FB)$ are minimal expressions of $t\in X \otimes F$ and their lengths are greater than $1$. Then there exists a bijection $\phi\colon A \to B$ such that $f=g\circ \phi$ and $\phi \cdot \sigma = \tau$.
    \[
    \begin{tikzcd}[row sep = 5pt]
        &A\ar[dd,"\phi", "{\rotatebox{90}{$\simeq$}}"']\ar[ld,"f"']&FA\ni \sigma \ar[dd,"F\phi", "{\rotatebox{90}{$\simeq$}}"']\\
        X&&\\
        &B\ar[lu,"g"]&FB\ni \tau
    \end{tikzcd}
    \]
\end{proposition}
\begin{proof}
    By \cref{LemmaMinimalInjectivity} and \cref{LemmaWeakUniquenessMinimalExpression}, there exists a unique bijection $\phi \colon A \to B$ such that $f=g\circ \phi$. 
    By definition, $g\otimes \tau=f\otimes \sigma=(g\circ \phi)\otimes \sigma =g\otimes(\phi \cdot \sigma) $. As $B$ is non-empty and $g$ is injective, \cref{LemmaCancellationInjection} completes the proof.
    \memo{iine!}
\end{proof}

\begin{example}[Model case 1: $F= \FinSet(S,-)$]
    If the functor $F$ is representable by a finite set $S$, every expression $(X\lar{f}A,A\lar{g}S)$ can be reduced to the form of
    \[
    (X\lar{f}A)\otimes (A\lar{g}S)= (X\lar{f\circ g}S)\otimes (S\lar{\id_{S}}S) =(X\lar{}\Image(f\circ g))\otimes (\Image(f\circ g)\lar{}S).
    \]
    By a straightforward argument, we can conclude that the expression $(X\lar{}\Image(f\circ g), \Image(f\circ g)\lar{}S)$ is the (essentially unique) minimal expression, and its length is $\# \Image(f\circ g)$. 
\end{example}

\begin{example}[Model case 2: $F= \P$ covariant power set functor]
    If the functor $F$ is the covariant power set functor $\P$, every expression $(X\lar{f}A,S\subset A)$ can be reduced to the form of
    \[
    (X\lar{f}A)\otimes (S\subset A) = (X\lar{f\restriction_{S}}S)\otimes (S\subset S) = (X\lar{}f(S))\otimes (f(S)\subset f(S)). 
    \]
    By a straightforward argument, we can conclude that the expression $(X\lar{}f(S), f(S)\subset f(S))$ is the (essentially unique) minimal expression, and its length is $\# f(S)$.  If the set $X$ is an infinite set, the length of elements of $X\otimes F$ is not bounded.
\end{example}

\begin{example}[Exceptional case, when $\Le(t)=1$]\label{ExampleINE}
    Let $F$ be a functor $F\colon \FinSet \to \Set$ that sends a finite set $A$ to
    \[
    FA \coloneqq 
    \begin{cases}
        \emptyset & (A\text{ is empty})\\
        \{\ast\} & (A\text{ is non-empty}).
    \end{cases}
    \]
    If $X$ is a non-empty and finite set, the set $X\otimes F$ is also non-empty, since $(X\lar{\id}X)\otimes(\ast\in FX) \in X\otimes F$. \memo{X: finite tsuiki si ta}
    Take an arbitrary element $t\in X\otimes F$ and an arbitrary expression $(X\lar{f}A,\ast\in FA)$. 
For any $x\in X$, we have
    \[
    t=(X\lar{f}A)\otimes (\ast\in FA) = (X\lar{\id}X)\otimes (\ast\in FX) = (X\lar{x}[1])\otimes (\ast\in F[1]).
    \]
    This implies that the length $\Le(t)$ is at most $1$. Since $F[0] = \emptyset$, the length $\Le(t)$ is actually $1$. If $\# X>1$, this example shows that the minimal expression may not be unique when $\Le(t)=1$.
\end{example}

\begin{example}[Exceptional case, when $\Le(t)=0$] \label{ExampleINE_2}
    Let $F$ be a functor $F\colon \FinSet \to \Set$ that sends a finite set $A$ to
    \[
    FA \coloneqq 
    \begin{cases}
        \{a,b\} & (A\text{ is empty})\\
        \{\ast\} & (A\text{ is non-empty}),
    \end{cases}
    \]
    where the two elements $a,b$ are distinct. 
For a non-empty set $X$ and an element $x\in X$, the element $t\coloneqq (X\lar{x}1)\otimes (\ast\in F[1]) \in X\otimes F$ has the following two minimal expressions:
    \[
    (X\lar{}[0])\otimes (a\in F[0]) = (X\lar{x}1)\otimes (\ast\in F[1]) =  (X\lar{}[0])\otimes (b\in F[0]). 
    \]
    This example shows that the minimal expression may not be unique when $\Le(t) =0$.
\end{example}

\subsubsection{The case of presheaves}

Utilizing \cref{PropositionFunctorTensorDescription}, for a functor $\X \colon \C \to \Set$ and an object $F\in \OC$, the length and minimal expression of an element $t \in (\X\otimes F)(c)$, where $c\in \ob(\C)$, are also defined as
those
in  $\X(c)\otimes F$.

\begin{lemma}\label{LemmaPresfeafDecreasing}
    For a functor $\X\colon \C \to \Set$ from a small category $\C$, and an object $F\in \OC$, the action of a morphism $g\colon c \to c'$ in $\C$ does not increase the length of an element $t \in \X(c)\otimes F$\textup{:}
    \[
    \Le(t)\geq \Le(g\cdot t).
    \]
\end{lemma}
\begin{proof}
    Let $(\X(c)\lar{f} A)\otimes (\sigma\in FA)$ be a minimal expression of the element $t$. Then, the element $g\cdot t \in \X(c')\otimes F$
    is equal to 
    \[
    g\cdot t=
    (\X(g) \otimes F)(t) = (\X(c') \lar{\X(g)} \X(c) \lar{f} A)\otimes \sigma.
    \]
    This proves that the length of 
    $g\cdot t \in \X(c')\otimes F$ is equal to or less than $\# A =\Le(t)$.
\end{proof}

\subsection{Length and minimal expression in the classifying topos of the theory of inhabited objects}

\subsubsection{Definition and basic properties}

As seen in \cref{PropositionWeakUniquenessMinimalExpression,ExampleINE,ExampleINE_2}, there are exceptional behaviors around the length $=0,1$. To avoid these exceptional cases, we will not consider the whole of the topos $\OC$, but its subtopos $\iOC$.

    \begin{notation}
    We adopt the following terminology:
    \begin{itemize}
        \item The category of non-empty finite sets is denoted by $\neFinSet$.
        \item We write $\iOC$ for the functor category $\iOC\coloneqq \Func{\neFinSet}{\Set}$, which is known to be \demph{the classifying topos of the theory of inhabited objects}\footnote{The notation $\iOC$ is also borrowed from \cite{anel2021topo}.}. (See \Cref{TheoremInhabitedObjectClassifier} and \Cref{rem:AsClassifyingTopos} for its universality.)
        \item The \demph{universal inhabited object} $\uniinhobj$ of the topos $\iOC$  is the canonical embedding functor $\neFinSet\to \Set$.
    \end{itemize}
    \end{notation}

    We will first observe how the topos $\iOC$ is embedded as a subtopos into the topos $\OC$.

    According to \cite{Maclane1998CWMcategories}, a functor $F\colon \C\to \D$ is called \demph{initial} if for any object $d \in \ob(\D)$, the comma category $\F\downarrow d$ is (non-empty and) connected.
    \begin{lemma}\label{Lemmainitiality}
        We define a functor from (the free category over) the parallel edge graph
        \[
        \begin{tikzcd}
            1\ar[r,shift left]\ar[r,shift right]&2
        \end{tikzcd}
        \]
        to the category $\neFinSet$ by sending objects $1$ (resp. $2$) to an $1$-element (resp. $2$-element) set and edges to two distinct functions.
        This functor is initial.
    \end{lemma}

    \begin{proposition}[$\iOC$ as a subtopos of $\OC$]\label{PropositioniOCinOC}
        The canonical embedding functor $\neFinSet\to \FinSet$ induces a geometric embedding $\iota \colon \iOC \hookrightarrow \OC$. The left adjoint $\iota^{\ast}$ is the restriction along the embedding $\neFinSet\to \FinSet$. The right adjoint $\iota_{\ast}$ extends a functor $F\colon \neFinSet \to \Set$ to the functor $\iota_{\ast} F \colon \FinSet\to \Set$ by defining $(\iota_{\ast} F)[0]$ to be the equalizer of
        \[
        \iota_{\ast} F[0] \rightarrowtail F[1]\rightrightarrows F[2].
        \]
    \end{proposition}
    \begin{proof}
        The fact that the fully faithful functor induces a geometric embedding is a famous fact. See \cite{maclane1994sheaves} for example. The right adjoint $\iota_{\ast}$ is calculated by the pointwise right Kan extension formula. Due to the formula, we have $(\iota_{\ast}F)A$ is $FA$ if $A\neq [0]$, and $\iota_{\ast}F[0]$ is the limit $\lim_{A\in \FinSet}FA$. By \cref{Lemmainitiality}, the calculation of the limit is reduced to the equalizer.
    \end{proof}


    \begin{definition}
    An object $X$ of a topos $\E$ is \demph{inhabited} (or internally inhabited) if the unique morphism to the terminal object $1$
 \[
 X \to 1
 \]
  is an epimorphism.
\end{definition}

 \begin{example}[Inhabited object in a presheaf topos]
        For a small category $\C$, a presheaf $F\colon \C^{\op}\to \Set$ is inhabited if and only if each $Fc$ is non-empty for every $c\in \ob(\C)$.
    \end{example}

\begin{remark}\label{rem:AsClassifyingTopos}
    This terminology is based on the fact that an object $X$ is inhabited if and only if the formula $\exists x \in X, \top$ holds internally. (Note that this is not equivalent to the non-emptiness condition $\neg\neg(\exists x \in X, \top)$.)  
    
    The fact that the classifying topos $\iOC$ is a subtopos of the classifying topos $\OC$ (\Cref{PropositioniOCinOC}) reflects that the theory of inhabited objects is obtained by adding the axiom $\exists x \in X, \top$ to the theory of objects (cf. \cite[Theorem 3.2.5]{caramello2018theories}). Indeed, this embedding $\iOC \hookrightarrow \OC$ is of the form of \cite[Theorem 1.1]{beke2004theories} or \cite[Theorem 8.2.10]{caramello2018theories}.
\end{remark}

    \begin{theorem}\label{TheoremInhabitedObjectClassifier}
        For any Grothendieck topos $\F$, the functor
        \[
            \Geom(\F, \iOC) \to \F \colon (f_{\ast}, f^{\ast}) \mapsto f^{\ast}(\uniinhobj)
        \]
        defines a fully faithful functor. Furthermore, this functor provides an equivalence of categories between $\Geom(\F, \iOC)$ and the full subcategory of $\F$ consisting of all inhabited objects of $\F$
        \[
        \Geom(\F, \iOC) \simeq \F_{\text{inhabited}},
        \]
        which restricts the equivalence of \cref{FactEquivObjectClassifier}
        \[
        \begin{tikzcd}
            \Geom(F,\iOC)\ar[r,"\simeq"]\ar[d,"{\iota}\circ {-}",hook]& \F_{\text{inhabited}}\ar[d,hook]\\
            \Geom(F,\OC)\ar[r,"\simeq"]& \F.
        \end{tikzcd}
        \]
    \end{theorem}
    \begin{proof}
        As the embedding $\iota \colon \iOC \hookrightarrow \OC$ is an embedding, there is a fully faithful embedding functor 
        \[
        \Geom(\F, \iOC)\to \Geom(\F,\OC)\simeq \F.
        \]
        We will prove that the essential image of the above functor is the full subcategory consisting of inhabited objects of $\F$.
        Take an arbitrary object $X\in \ob(\F)$. Let $f\colon \F \to \OC$ denote the corresponding geometric morphism.
        Under these settings, the following conditions are equivalent:
        \begin{itemize}
            \item The geometric morphism $f$ factors through the geometric embedding $\iOC\hookrightarrow \OC$.
            \item 
        $\iff$ 
        the direct image functor
        \[
        f_{\ast}\colon \F \to \OC \colon A \mapsto (n\mapsto \F(X^n,A)).
        \]
        factors through the embedding functor 
        $\iota_{\ast}\colon \iOC\hookrightarrow \OC$.
        \item
        $\iff$(\cref{PropositioniOCinOC})
        For every $A\in \ob(\F)$,
        the set of global sections $\F(X^0, A) \cong \F(1, A) \cong \Gamma(A)$ is (naturally) isomorphic to the equalizer 
        \[
        \begin{tikzcd}
            \F(1,A)\ar[r]&\F(X,A)\ar[r,shift left]\ar[r,shift right]&\F(X^2,A)
        \end{tikzcd}
        \]
        in the category of sets $\Set$.
        \item
        $\iff$
        The diagram
        \[
        \begin{tikzcd}
            X^2\ar[r,shift left]\ar[r,shift right]&X\ar[r]&1
        \end{tikzcd}
        \]
        is a coequalizer diagram in the topos $\F$.
        \item 
        $\iff$ (Topoi are regular.)
        The unique morphism $X\to 1$ is an epimorphism.
        \end{itemize}
        This completes the proof.
    \end{proof}
This theorem is mentioned in \cite{anel2021topo}.



\subsubsection{Length and minimal expression}\label{SubsectionLengthInhabitedcases}

Now we can prove that if $F$ is in the essential image of the embedding $\iota\colon \iOC \to \OC$, then \cref{LemmaWeakUniquenessMinimalExpression} and \cref{PropositionWeakUniquenessMinimalExpression} 
hold with no exception. 

\begin{lemma}[Inhabited version of Comparison lemma (\cref{LemmaWeakUniquenessMinimalExpression})]\label{LemmaStrongUniqunessMinimalExpression}
    Suppose $(X\lar{f} A, \sigma \in FA)$ is a minimal expression of $t\in X\otimes F$ and $F$ is in the essential image of $\iota_{\ast} \colon \iOC \to \OC$. Then, for any other expression $(X\lar{g} B, \tau \in FB)$, we have $\Image(f) \subset \Image(g)$. 
\end{lemma}
\begin{proof}
    If $\# A$ is more than one, this statement follows from \cref{LemmaWeakUniquenessMinimalExpression}. If $A =\emptyset$, the statement is trivial. So we assume $\Image(f)=\{x_0\}\not\subset \Image(g)$ and deduce contradiction. 

    As $f\otimes \sigma$ is a minimal expression, $B$ is non-empty. So, take $x_1\in \Image(g)$ and let $h\colon X \to X$ be a function defined by 
    \[
    h(x)=
    \begin{cases}
        x_1 & (x=x_0)\\
        x & (\text{otherwise}).
    \end{cases}
    \]
    Then, the composite of $X\lar{h} X \lar{g} B$ is equal to $g$. Applying $h\otimes F$ (\cref{LemmaNaturalityOfCoend}), the equation $f\otimes \sigma = g\otimes \tau$ implies the equation $(h\circ f)\otimes \sigma = (h\circ g)\otimes \tau=g\otimes \tau=t=f\otimes \sigma$. 
    Let $s=f\coprod (h\circ f)\colon A\coprod A\to X$ be the coproduct of $f$ and $h\circ f$, and $\iota_0,\iota_1\colon A\to A\coprod A$ be natural inclusions. Then, 
    \[
    s\otimes (\iota_0 \cdot \sigma)=(s\circ \iota_0) \otimes \sigma= f\otimes \sigma=(h\circ f)\otimes \sigma=(s\circ \iota_1)\otimes \sigma=s\otimes (\iota_1 \cdot \sigma).
    \]
    As $\Image(s)=\Image(f)\cup \Image(h\circ f)=\{x_0,x_1\}$, $s$ is injective, so $\iota_0 \cdot \sigma=\iota_1 \cdot \sigma$ by \cref{LemmaCancellationInjection}. As $F$ is in the essential image of $\iota_{\ast}\colon \iOC \to \OC$, 
        \[
        \begin{tikzcd}
            F[0]\ar[r]&FA\ar[r,shift left]\ar[r,shift right]&F(A\coprod A)
        \end{tikzcd}
        \]
    is an equalizer diagram. Therefore, we can take $\varepsilon \in F[0]$ such that $(A \lar{} [0])\cdot (\varepsilon \in F[0]) = \sigma$. Then, $t=f\otimes \sigma = (X\lar{} [0])\otimes (\varepsilon \in F[0])$ and it contradicts the fact that $f\otimes \sigma$ is a minimal expression of $t$.
\end{proof}
    
\begin{proposition}[Inhabited version of Uniqueness of minimal expression (\cref{PropositionWeakUniquenessMinimalExpression})]\label{PropositionStrongUniquenessMinimalExpression}
    Suppose $(X\lar{f} A, \sigma \in FA)$ and $(X\lar{g} B, \tau \in FB)$ are minimal expressions of $t\in X\otimes F$ , $F$ is in the essential image of $\iota_{\ast}\colon \iOC\to \OC$ and $X\neq \emptyset$. Then there exists a bijection $\phi\colon A \to B$ such that $f=g\circ \phi$ and $\phi \cdot \sigma = \tau$.
    \[
    \begin{tikzcd}[row sep = 5pt]
        &A\ar[dd,"\phi", "{\rotatebox{90}{$\simeq$}}"']\ar[ld,"f"']&FA\ni \sigma \ar[dd,"F\phi", "{\rotatebox{90}{$\simeq$}}"']\\
        X&&\\
        &B\ar[lu,"g"]&FB\ni \tau
    \end{tikzcd}
    \]
    
\end{proposition}

\begin{proof}    
    First, we consider the case the length of $t$ is not zero. By \cref{LemmaStrongUniqunessMinimalExpression},  $\Image(f)=\Image(g)\neq \emptyset$. Therefore, the same argument as \cref{PropositionWeakUniquenessMinimalExpression} shows the existence of $\phi$. 
    
    Then, we consider the case the length of $t$ is zero. In this case, $A=B=\emptyset$ and $f,g$ are the empty map. Let $F[1]\ni\sigma'=([1] \lar{} [0])\cdot (\sigma \in F[0])$ and $F[1]\ni\tau'=([1] \lar{} [0])\cdot (\tau \in F[0])$. As $X\neq\emptyset$, we can take a morphism $h\colon [1]\to X$. Then, 
    \[
    (X\lar{h} [1])\otimes( \sigma' \in F[1])=(X\lar{f} [0])\otimes( \sigma \in F[0])=t=(X\lar{g} [0])\otimes( \tau \in F[0])=(X\lar{h} [1])\otimes( \tau' \in F[1]).
    \]
    Therefore, we obtain $\sigma'=\tau'$ by \cref{LemmaCancellationInjection}. As
    \[
        \begin{tikzcd}
            F[0]\ar[r]&F[1]\ar[r,shift left]\ar[r,shift right]&F[2]
        \end{tikzcd}
    \]
    is an equalizer diagram, $F[0]\to F[1]$ is injective and $\sigma =\tau$. As noted before, $f$ and $g$ are empty maps. So, $\phi=\id_{\emptyset}$ satisfies the condition of this proposition. 
\end{proof}

\section{Rigidities of an object}\label{SectionRigidities}

\begin{definition}[Rigidity]
    An object $X$ of a category $\C$ is \demph{rigid} if the identity morphism $\id_{X}\colon X \to X$ is the unique endomorphism of $X$.
\end{definition}



\begin{definition}[Lex-rigidity]
    An object $X$ of a category with finite limits $\C$ is \demph{lex-rigid}, if for any $n\geq 0$ and a morphism $f\colon X^n \to X$, there uniquely exists $0\leq i<n$ such that $f$ is equal to the $i$-th projection $p_i\colon X^n \to X$.
\end{definition}


\begin{definition}[Inhabited-lex-rigidity]
    An object $X$ of a category with binary products $\C$ is \demph{inhabited-lex-rigid}, if for any $n>0$ and a morphism $f\colon X^n \to X$, there uniquely exists $0\leq i<n$ such that $f$ is equal to the $i$-th projection $p_i\colon X^n \to X$.
\end{definition}

\begin{definition}[Topos-rigidity]
    An object $X$ of a Grothendieck topos $\E$ is \demph{topos-rigid} if the lex cocontinuous functor $\gmX\colon \OC \to \E$ that sends the universal object $\uniobj$ to $X$ is fully faithful (i.e., the corresponding geometric morphism $\E \to \OC$ is connected and defines a quotient topos).
\end{definition}
\memo{The free lex-cocomplete category with one object might be different from $\OC$...? See \cite{anel2021topo}.(We avoided this problem by considering only Grothendieck topoi.)}
\begin{definition}[Inhabited-topos-rigidity]
    An object $X$ of a Grothendieck topos $\E$ is \demph{inhabited-topos-rigid} if the object $X$ is inhabited and the corresponding geometric morphism $\E \to \iOC$ is connected.
\end{definition}


We are interested in the notion of rigidities because we can use it to solve Lawvere's open problem \cite{Open240411LawvereRevised}.


\begin{lemma}\label{LemmaMetaRigidity}
    If an endofunctor $F\colon \iOC \to \iOC$ is an equivalence of categories, then $F$ is naturally isomorphic to $\id_{\iOC}$.
\end{lemma}
\begin{proof}
Due to the universality of the topos $\iOC$, it is enough to prove that the functor $F$ sends the universal inhabited object $\uniinhobj$ to itself up to isomorphism. To prove this, we characterize the object $\uniinhobj$ by the categorical structure of the topos $\iOC$. Since the category $\neFinSet$ is Cauchy complete, (the essential image of) the full subcategory $(\neFinSet)^{\op}\hookrightarrow \iOC$ is characterized as the full subcategory consisting of all tiny (= small projective) objects. Therefore, the universal inhabited object $\uniinhobj$ is characterized as an initial object among the tiny objects in $\iOC$. This characterization is preserved by the equivalence $F$.
\end{proof}
\begin{proposition}\label{PropositionIsotopoiImpliesIsoobjects}
    If a Grothendieck topos $\E$ has proper class many (mutually non-isomorphic) inhabited-topos-rigid objects, then the topos $\E$ has proper class many (mutually non-equivalent) quotient topoi.
\end{proposition}
\begin{proof}
    For two inhabited-topos-rigid object $X,Y\in \ob(\E)$, the corresponding geometric morphisms $\overline{X}, \overline{Y}\colon \E \to \iOC$ are connected, and thus define quotient topoi of $\E$. 

    Suppose these two quotient topoi are equivalent, i.e., there exists an equivalence of categories $F\colon \iOC \to \iOC$ such that the following diagram of geometric morphisms
    \[
    \begin{tikzcd}[row sep = 10pt]
        &\iOC\ar[dd,"F"{name = T}]\\
        \E\ar[ru,"\overline{X}"]\ar[rd,"\overline{Y}"']\ar[to = T, phantom, "\cong"]&\\
        &\iOC
    \end{tikzcd}
    \]
    is commutative up to natural isomorphism.
    By \cref{LemmaMetaRigidity}, we conclude that two geometric morphisms $\overline{X}, \overline{Y}$ are naturally isomorphic, and especially $X,Y$ are isomorphic by the Yoneda lemma\memo{tuiki sitaked jyoutyoudattara kesite kudasai}. This proves that the two non-isomorphic inhabited-topos-rigid objects induce two non-equivalent quotient topoi. 
\end{proof}

Now our task is reduced to constructing a Grothendieck topos that has proper class many mutually non-isomorphic inhabited-topos-rigid objects. For convenience, we rephrase the inhabited-topos-rigidity in elementary terms.


\begin{lemma}\label{LemmaiOCLeftDescription}
    For a Grothendieck topos $\E$ and an inhabited object $X\in \ob(\E)$, let $f\colon \E \to \iOC$ be the corresponding geometric morphism. Then the left adjoint part $f^{\ast}$ is naturally isomorphic to the composite of the embedding $\iota_{\ast}\colon \iOC \to \OC$ and $X\otimes {-}\colon \OC \to \E$:
    \[
    \begin{tikzcd}
        \E && \iOC\ar[ll,"f^{\ast}"', ""{name=A}]\ar[ld,"\iota_{\ast}"]\\
        &\OC\ar[lu,"X\otimes {-}"]\ar[to = A, phantom, "\cong"] &
    \end{tikzcd}
    \]
\end{lemma}
\begin{proof}
Let $g\colon \E \to \OC$ denote the geometric morphism corresponding to the object $X$, whose left adjoint part $g^{\ast}$ is $X\otimes {-}$ (\cref{NotationTensorFunctor}).
    The geometric morphism $f\colon \E \to \iOC$ is a codomain restriction of the geometric morphism $g\colon \E\to \OC$, i.e., $g\cong \iota \circ f$. Then we have $g^{\ast}\cong f^{\ast} \circ \iota^{\ast}$. By precomposing $\iota_{\ast}$, we obtain $g^{\ast}\circ \iota_{\ast} \cong f^{\ast} \circ \iota^{\ast}\circ \iota_{\ast} \cong f^{\ast}$.
\end{proof}

\begin{lemma}\label{LemmaRepresentablesAreInIOC}
    For any finite set $S$, the representable functor $y(S)=\FinSet(S,{-})\colon \FinSet \to \Set$ is in the essential image of the embedding $\iota_{\ast}\colon \iOC \to \OC$.
\end{lemma}
\begin{proof}
    It suffices to prove that the following diagram is an equalizer diagram:
    \[
    \begin{tikzcd}
        \FinSet(S,[0]) \ar[r]&\FinSet(S,[1]) \ar[r, shift left]\ar[r, shift right]&\FinSet(S,[2]).
    \end{tikzcd}
    \]
\end{proof}

\begin{proposition}\label{PropositionUnwindingInhabitedToposRigidity}
For a Grothendieck topos $\E$ and an inhabited object $X\in \ob(\E)$, the object $X$ is inhabited-topos-rigid if and only if, for any 
\begin{itemize}
    \item object $F\in \ob(\OC)$ in the essential image of $\iota_{\ast}\colon \iOC \to \OC$,
    \item positive integer $n>0$, and
    \item morphism $f\colon X^n \to X\otimes F$ in the topos $\E$,
\end{itemize}
there exists a unique element $\tau \in F[n]$ such that $f= {-}\otimes \tau$ (\cref{NotationOtimesSigmaMorphism}).
\end{proposition}
\begin{proof}
Let $f$ denote the geometric morphism $f\colon \E \to \iOC$ corresponding to the inhabited object $X$. 
    The following conditions are equivalent:
    \begin{itemize}
        \item The object $X$ is inhabited-topos-rigid.
        \item $\iff$ for any pair of objects $G,F \in \ob(\iOC)$, the function
        \[
        f^{\ast}\colon \iOC(G,F)\to \E(f^{\ast} G, f^{\ast} F)
        \]
        is bijective.
        \item $\iff$(Codensity theorem) for any object $F \in \ob(\iOC)$, a positive integer $n>0$, the function
        \[
       f^{\ast}\colon \iOC(y[n],F)\to \E(f^{\ast} y[n], f^{\ast} F)
        \]
        is bijective.
        \item $\iff$(\cref{LemmaiOCLeftDescription} and $\iota_{\ast}$ is fully faithful) for any object $F \in \ob(\iOC)$, and a positive integer $n>0$, the function
        \[
        X\otimes {-}\colon \OC(\iota_{\ast} y[n],\iota_{\ast} F)\to \E(f^{\ast} y[n], f^{\ast} F)
        \]
        is bijective.
        \item $\iff$ (\cref{LemmaRepresentablesAreInIOC} and the Yoneda lemma)
        for any object $F \in \ob(\OC)$ in the essential image of the embedding $\iota_{\ast}\colon \iOC \to \OC$, and a positive integer $n>0$, the function
        \[
         F[n]\ni \tau  \mapsto (-\otimes \tau)\in\E(X^n, X\otimes  F)
        \]
        is bijective.
    \end{itemize}
\end{proof}
\section{Lex-rigid objects in the category of relational structures}\label{SectionLexRigidityAndLStructures}

\begin{definition}We adopt the following definitions:
\begin{itemize}
    \item A \demph{relational language} $\L$ is a language only with relational symbols, i.e., a set of relational symbols $\L=\{Q_{\lm}\}_{\lm \in \Lm}$ equipped with the arity $n_\lm \in \Z_{\geq 0}$.
    \item An \demph{$\L$-structure} $\A$ is a set $A$ equipped with a family of subsets $Q_{\lm}^A \subset A^{n_\lm}$. 
    \item For two $\L$-structures $\A=(A, Q_{\lm}^A), \B = (B,Q_{\lm}^B)$, an \demph{$\L$-structure morphism} $f\colon \A\to \B$ is a function between the underlying sets $f\colon A \to B$ such that $f(Q_{\lm}^A)\subset Q_{\lm}^B$ for every $\lm \in \Lm$.
    \item The category of $\L$-structures and $\L$-structure morphisms is denoted by $\Lstr$.
\end{itemize}
\end{definition}

Note that the category $\Lstr$ has all small products. In fact, the product of the $\L$-structures $(\A_i)_{i\in I}=(A_i,Q_{{\lm}}^i)_{i\in I}$ is the $\L$-structure $\A:=(\prod_{i\in I} A_i, \prod_{i\in I} Q_{\lm}^i)$. In particular, $n$-th power of $\A=(A,Q_{\lm}^A)$ is $(A^n,(Q_{\lm}^A)^n)$.


The goal of this section is to prove this proposition:
\begin{proposition}\label{PropositionManyNonEmptyLexRigid}
    There exists a countable relational language $\L=\{Q_{\lm}\}_{\lm\in\Lm}$ which has no nullary relational symbols, and for every cardinal $\kappa$, there exists an $\L$-structure $(X,Q_{\lm})$ satisfying the following conditions:
    \begin{itemize}
    \item $(X,Q_{\lm})$ is inhabited-lex-rigid in $\Lstr$.
    \item For all $\lm \in \Lm$, $Q_{\lm}\neq \emptyset$.
    \item $(\#X)^{\aleph_0}=\#X\geq \kappa$. 
    \end{itemize}
    
\end{proposition}
\begin{remark}\memo{kesu?}
    Let $S$ be a class that consists of $\L$-structures. Then, the following conditions are equivalent:
    \begin{itemize}
        \item $S$ has proper class many isomorphism classes.
        \item for any cardinal $\kappa$, there exists $(X,Q_{\lm}^X)\in S$ such that $\#X\geq \kappa$. 
    \end{itemize}
    Indeed, for fixed $\kappa$, there are up to $\aleph_0 2^{\kappa(\#\L)}$ $\L$-structures whose (underlying set's) size is less than $\kappa$, up to isomorphism.
\end{remark}

For rest of this section, $R_2$ denotes binary relation symbol. 

\begin{proposition}\label{PropositionManyRigid}
    For any infinite cardinal $\kappa$, there exists $\{R_2\}$-structure $(X,R_2^X)$ satisfying the following conditions:
    \begin{itemize}
    \item $(X,R_2^X)$ is rigid in $\mathbf{Str}_{\{R_2\}}$.
    \item $R_2^X \neq\emptyset$. 
    \item $(\#X)^{\aleph_0}=\#X\geq \kappa$. 
    \end{itemize}
\end{proposition}
\begin{proof}
    According to \cite{vopvenka1965rigid} (or \cite[Lemma 2.64]{adamek1994locally}), for any set $X$, there exists $R_2^X \subset X\times X$ such that $(X,R_2^X)$ is a rigid $\{R_2\}$-structure. The second condition is automatically satisfied, as $(X, \emptyset)$ is not rigid if $X$ is infinite. The third condition is satisfied if we take $X$ to be of size $\kappa^{\aleph_0}$. 
\end{proof}

We introduce the notation of pairing functions for later reference. 
\begin{definition}
    Let $X$ be a set and $(p_i\colon X\to X)_{i\in I}$ be a system of endomorphisms. We say that $(p_i)_{i\in I}$ is a \textbf{pairing system} if 
    \[
    \begin{tikzcd}[column sep =100pt]
        X\ar[r,"\prod_{i\in I} p_i"]& X^I
    \end{tikzcd}
    \]
    is a bijection. 
\end{definition}
\begin{notation}\label{NotationPairing}
    Assume  $(p_i\colon X\to X)_{i\in [n]}$ is a pairing system. 
    For an $n$-tuple of elements $(x_0, \dots, x_{n-1})\in X^n$, the corresponding element via the above bijection
    is denoted by 
    $\begin{pmatrix}
        x_0 \\
        \vdots\\
        x_{n-1}
    \end{pmatrix} \in X$. Tautologically, the function $p_i$ sends an element $\begin{pmatrix}
        x_0 \\
        \vdots\\
        x_{n-1}
    \end{pmatrix}$ to $x_i$.
\end{notation}

\begin{proof}[Proof of \cref{PropositionManyNonEmptyLexRigid}]
    For each positive integer $n$, let $S_n$ be a ternary relational symbol and $T_n$ be a unary relational symbol. We define $\L_0$ as $\{R_2\}$ and $\L$ as $\{R_2\} \cup \{S_n\}_{n\in \Z_{>0}} \cup \{T_n\}_{n\in \Z_{>0}}$.  
    We will show that $\L$ satisfies the conditions. By definition, $\L$ has no nullary relational symbol. 
    Take arbitrary infinite cardinal $\kappa$, and assume $\mathcal{X}_0=(X,R_2^X)\in\mathbf{Str}_{\L_0}$ satisfies the condition in \cref{PropositionManyRigid}. We will construct an $\L$-structure on $X$. 
    
    As $\#X\geq \kappa\geq \aleph_0$, we can take pairwise distinct elements $c_0,c_1,\dots \in X$. 
    By the same reason, we can take a family of endomorphisms $(p^n_i\colon X\to X)_{0\leq i<n<\omega}$ such that, for any positive integer $n$, the subfamily $(p^n_i)_{i\in [n]}$ is a pairing system of $X$. Then, we define $T_n^X$ as $\{c_0,c_1,\dots, c_{n-1}\}$ and $(x,y,z)\in S_n^X$ if and only if $y=c_i$ for some $i<n$ and $p^n_i(x)=z$. 
    We prove $\mathcal{X}=(X,R_2^X,S_n^X,T_n^X)$ is an inhabited-lex-rigid object in $\mathbf{Str}_{\L}$. 
    
    Let $n$ be a positive integer and $f\colon \mathcal{X}^n\to \mathcal{X}$ be a morphism of $\mathbf{Str}_{\L}$. We write $c$ for $f(c_0,c_1,\dots, c_{n-1})$. As $f$ preserves $T_n^X$, we have $c\in T_n^X$ and hence $c=c_{i_0}$ for some $i_0<n$. 
    Take an arbitrary $n$-tuple $(x_0,x_1,\dots, x_{n-1}) \in X^n$ and let $x\in X $ be the unique element such that $p^n_i(x)=x_i$ for all $i<n$. As $\mathcal{X}_0=(X,R_2^X)$ is rigid in $\mathbf{Str}_{\L_0}$, the morphism $\mathcal{X}_0 \overset{\Delta}{\to} \mathcal{X}_0^n \overset{f}{\to} \mathcal{X}_0$ must coincide with the identity function, in particular, we obtain $f(x,x,\dots, x)=x$.  As $f$ preserves $S_n^X$, we have $(x,c_{i_0},f(x_0,\dots, x_{n-1}))=(f(x,\dots,x),f(c_0,\dots, c_{n-1}),f(x_0,\dots ,x_{n-1}))\in S_n^X$, which implies $f(x_0,\dots, x_{n-1})=p^n_{i_0}(x)=x_{i_0}$. As we took $x_0,\dots, x_{n-1}$ arbitrarily, this shows $f$ is the $i_0$-th projection and $X$ is inhabited lex-rigid. 
    
    The second condition of $\mathcal{X}$ follows from $R_2^X\neq \emptyset$, $c_0\in T_n^X$ and $(c_0,c_0,p_0^n(c_0))\in S_n^X$. The third condition is trivial. 
\end{proof}

\section{Presheaf encoding of relational structures}\label{SectionPresheafEncodingAndMainTheorem}
    In this section, $\N$ denotes the set of non-negative integers and $\overline{\N}$ denotes $\N\cup \{\omega\}$. This set is equipped with the total order $0<1<2<3<\dots<\omega$. 
    In this section, let $\L$ be a relational language, i.e., a (small) set of relational symbols $\L=\{Q_\lm\}_{\lm\in \Lm}$. The arity of the relational symbol $Q_{\lm}$ is denoted by $n_{\lm}\in \N$. We assume $\L$ has no nullary relational symbols; i.e., for all $\lm\in \Lm$, $n_{\lm}\neq 0$.\memo{Setting ni iretai}

\begin{definition}\label{DefinitionPresheafEncoding}
    The \demph{encoding graph} of the language $\L$ is the directed graph $(V,E)$, where
    \begin{itemize}
        \item the set of vertices $V$ is the disjoint union of (the singleton of) the \demph{main vertex} $\{v\}$ and the set 
        $\{v_{\lm}\}_{\lm\in \Lm}$,  and
        \item the set of edges $E$ is the union of
        \begin{itemize}
            \item $\{p^n_i\colon v \to v\}$ for $0\leq i <n\leq\omega$.
            \item $\{p'^n_i\colon v \to v\}$ for $0\leq i <n<\omega$.
            \item $\{s'_{\lm}\colon v_{\lm}\to v\}_{\lm\in \Lm}$
            \item $\{r_{\lm}\colon v\to v_{\lm}\}_{\lm\in \Lm}$
        \end{itemize}
    \end{itemize}

    The \demph{encoding topos} $\E_{\L}$ of the language $\L$ is a full subcategory of the covariant functor category $ [(V,E), \Set]$ from (the free category over) the encoding graph to the category $\Set$. A functor $\X \colon (V,E) \to \Set$ belongs to $\E_{\L}$ if and only if the action of $r_\lm$ is a retract of the action of $s'_{\lm}$
    \[
    \begin{tikzcd}
        \X(v_{\lm})\ar[r, shift left,"s'_{\lm}", rightarrowtail]&\X(v)\ar[l,shift left, "r_{\lm}", twoheadrightarrow].
    \end{tikzcd}
    \]
\end{definition}

The encoding topos of a certain relational language will be proven to be the desired example of Grothendieck topos with proper class many quotient topoi. Now we emphasize it is a Grothendieck topos.
\begin{proposition}\memo{Ichiou Key nanode kaita zyoutyounara kesite iidesu}
    For a relational language $\L$, its encoding topos $\E_{\L}$ is a Grothendieck topos.
\end{proposition}
\begin{proof}
    Let $\C_{\L}$ be the category freely generated by the encoding graph $(V,E)$ and relations $\{r_{\lm} \circ s'_{\lm} = \id\}$.
    As $\L$ is assumed to be small, the category $\C_{\L}$ is small. This implies that the encoding topos $\E_{\L}$ is a presheaf topos over a small category $\E_{\L} \simeq \Func{\C_{\L}}{\Set}$. 
\end{proof}

\begin{notation}
    For $\X\in \ob(\E_{\L})$ and $f\in \mathrm{mor}(\E_{\L})$, we denote $\X(v)$ as $\underline{\X}$ and $f_v$ as $\underline{f}$. 
\end{notation}

\begin{definition}
    A \demph{presheaf encoding} of an $\L$-structure $(X, Q_\lm^X \subset X^{n_\lm}$) is an object of $\E_{\L}$, such that
    \begin{itemize}
        \item the value at the main vertex is the underlying set $X$,
        \item for each $\lambda\in \Lambda$, the value at the vertex $v_{\lambda}$ is the set $X\times Q_{\lambda}^{X}$,
        \item for any $0\neq n \in \overline{\N}$, $(p^n_i)_{i;i<n}$ is a pairing system,
        \item for any $n \in \N$ and $i<n$, 
        $p'^n_i$ is the composite of
        \[
            \begin{tikzcd}[column sep = 50pt]
            X\ar[r,"\prod_{0\leq i<2} p^2_i"]&X^2\ar[r,"\id_X\times p^n_i"]&X^2\ar[r,"(\prod_{0\leq i<2} p^2_i)^{-1}"]&X,
            \end{tikzcd}
        \]
        \item for any $\lm \in \Lm$, we define a function $s_{\lm} \colon Q_\lm^X \to X$ by the composite of 
        
        \[
        \begin{tikzcd}[column sep = 50pt]
            Q_\lm^X\ar[r,"\iota"]&X^{n_\lm}\ar[r,"(\prod_{0\leq i<{n_\lm}} p^{n_\lm}_{i})^{-1}"]&X, 
        \end{tikzcd}
        \]
        and 
        \item for any $\lm \in \Lm$, $s'_{\lm}$ is the composite of 
        \[
            \begin{tikzcd}[column sep = 50pt]&X\times Q_{\lm}^X\ar[r,"\id_X\times s_{\lm}"]&X^2\ar[r,"(\prod_{0\leq i<2} p^2_i)^{-1}"]&X
            \end{tikzcd}
        \]
        and $r_{\lm}$ is left inverse of $s'_{\lm}$. 
        
    \end{itemize}
\end{definition}

Notice that a presheaf encoding might not exist (depending on the cardinality of $X$), and even if it exists, it is not unique (even up to isomorphisms). If $(\#X)^{\aleph_0} = \# X$ and $Q_{\lm}^X \neq \emptyset$ for every $\lm \in \Lm$, then it has a presheaf encoding.


Also note that the same symbols, $p^n_i, p'^n_i, s'_{\lm}, r_{\lm}$, are used to refer to three different things: edges in the encoded graph, morphisms in the generated category $\C_{\L}$, and the corresponding actions on presheaf encodings. Since the remainder of this paper primarily concerns the actions, and does not directly deal with the edges and morphisms, we believe this abuse of notation should not lead to any substantial confusion.

\begin{setting}\label{SettingZero}
To prepare for later reference, we summarize the following setting:
    \begin{itemize}
        \item $\L=\{Q_{\lm}\}_{\lm \in \Lm}$ is a relational language in which each relational symbol $Q_{\lm}$ has positive arity $n_{\lm}>0$.
        \item $(X,\{Q_\lm^X \subset X^{n_\lm}\}_{\lm \in \Lm})$ is an $\L$-structure such that $\#X^{\aleph_0}=\#X \geq \aleph_0$ and $Q_{\lm}^X\neq \emptyset$ for all $\lm \in \Lm$. 
        \item $\X \in \ob(\E_{\L})$ is a presheaf encoding of the $\L$-structure $(X,\{Q_\lm^X \subset X^{n_\lm}\}_{\lm \in \Lm} )$.
    \end{itemize}
\end{setting}

\begin{notation}\label{NotationVertical}
     Using \cref{NotationPairing}, we can write the action of $p'^{n}_{i}$ as
\[
\begin{pmatrix}
x\\
\begin{pmatrix}
        x_0 \\
        \vdots\\
        x_{n-1}
\end{pmatrix} 
\end{pmatrix}
    \mapsto 
    \begin{pmatrix}
        x \\
        x_i
\end{pmatrix},
\]
 the function $s_{\lm}$ as 
\[
Q_{\lm}^X\ni(x_0,x_1,\dots x_{n_{\lm}-1})\mapsto \begin{pmatrix}
        x_0 \\
        \vdots\\
        x_{n_{\lm}-1}
\end{pmatrix} \in X,
\]
 and the action of $s'_{\lm}$ as
 \[
 X\times Q_{\lm}^X\ni(x,(x_0,x_1,\dots x_{n_\lm-1}))\mapsto
 \begin{pmatrix}
x\\
\begin{pmatrix}
        x_0 \\
        \vdots\\
        x_{n_\lm-1}
\end{pmatrix} 
\end{pmatrix}\in X.
\]
\end{notation}



\begin{setting}\label{SettingOne}
To prepare for later reference, we summarize the following setting:
    \begin{itemize}
        \item All settings in \cref{SettingZero}.
        \item $F$ is an object of the topos $\OC$, and $F$ is in the essential image of the embedding  $\iota_{\ast}\colon \iOC \to \OC$.
        \item $f\colon \X^n \to \X \otimes F$ is a morphism in the encoding topos $\E_{\L}$, where $n$ is a \textbf{positive} integer.
    \end{itemize}
\end{setting}
\begin{proposition}\label{ProppositionImagefBounded}
    Under the \cref{SettingOne}, the set
    \[
    \{\Le(\underline{f}(x))\mid x= (x_0, \dots x_{n-1})\in \underline{\X^n}\} \subset \N
    \]
    is bounded.
\end{proposition}
\begin{proof}
Suppose the set is not bounded. We will deduce the contradiction. We can take an infinite sequence $\{x_k = (x_{k,0} \dots x_{k,n-1}) \in \underline{\X^n}\}_{k\in \N}$ such that the length of $\underline{f}(x_k)$ is greater than $k$.

Since the function
        \[
        \begin{tikzcd}[column sep = 50pt]
            X\ar[r,"\prod_{i\in \N} p^\omega_i"]&X^\N
        \end{tikzcd}
        \]
        is bijective, we can take the unique element $x_{\omega,i}\in X$ such that $p^{\omega}_{k}(x_{\omega,i}) = x_{k,i}$ for each $0\leq i<n$. Consider the $n$-tuple $x_{\omega}=(x_{\omega,0}, \dots x_{\omega,n-1}) \in X^n$ and let $L$ be the length of the element $\underline{f}(x_{\omega})$. 

        \Cref{LemmaPresfeafDecreasing} and the naturality of $f$ at $p_{L}^{\omega}$ implies 
        \[
        L= \Le (\underline{f}(x_\omega)) \geq \Le( (p^{\omega}_{L} \otimes F) \underline{f}(x_\omega)) = \Le( \underline{f}((p^{\omega}_{L})^n  (x_\omega)))
        =
        \Le( \underline{f}(x_L)) > L.
        \]
        This is a contradiction.
\end{proof}



\begin{setting}\label{SettingAlpha}
For the later references, we summarize a list of settings:
\begin{itemize}
    \item All of \cref{SettingZero} $\subset$ \cref{SettingOne}. 
    \item Let $m$ denote the maximum length of $\underline{f}(\underline{\X^n})$. 
    \item Fix an element $(a_0, \dots, a_{n-1})\in X^n=\underline{\X^n}$ such that the length of $f(a_0, \dots, a_{n-1})$ is $m$.
    \item Fix a minimal expression  $(\alpha_0, \dots ,\alpha_{m-1})\otimes(\sigma \in F[m])$ of $\underline{f}(a_0, \dots, a_{n-1})$.
\end{itemize}
    
\end{setting}

\begin{proposition}\label{PropositionInducedFunction}
    Suppose the \cref{SettingAlpha}. 
    For any $n$-tuple of elements $(b_0, \dots, b_{n-1})$, the length of $\underline{f}\left(
    \begin{pmatrix}
        a_0\\
        b_0
    \end{pmatrix},
    \dots,
    \begin{pmatrix}
        a_{n-1}\\
        b_{n-1}
    \end{pmatrix}
    \right) $ is $m$. Furthermore, there uniquely exists $(\beta_0,\beta_1,\dots ,\beta_{m-1})\in X^m$ such that
    \[
    \underline{f}\left(
    \begin{pmatrix}
        a_0\\
        b_0
    \end{pmatrix},
    \dots,
    \begin{pmatrix}
        a_{n-1}\\
        b_{n-1}
    \end{pmatrix}
    \right) = 
    \left(
    \begin{pmatrix}
        \alpha_0\\
        \beta_0
    \end{pmatrix},
    \dots,
    \begin{pmatrix}
        \alpha_{m-1}\\
        \beta_{m-1}
    \end{pmatrix}
    \right)
    \otimes (\sigma \in F[m]).
    \]
\end{proposition}
\begin{proof}

    Let
    \[
    \left(
    \begin{pmatrix}
        \gamma_0\\
        \delta_0
    \end{pmatrix},
    \dots,
    \begin{pmatrix}
        \gamma_{l-1}\\
        \delta_{l-1}
    \end{pmatrix}
    \right)
    \otimes (\tau \in F[l])
    \]
    be a minimal expression of $\underline{f}\left(
    \begin{pmatrix}
        a_0\\
        b_0
    \end{pmatrix},
    \dots,
    \begin{pmatrix}
        a_{n-1}\\
        b_{n-1}
    \end{pmatrix}
    \right)$. By the definition of $m$, we have $l\leq m$.
    Applying $p^2_0$, we have
    \[
    (\alpha_0,\alpha_1,\dots ,\alpha_{m-1})\otimes \sigma= (\gamma_0,\gamma_1,\dots ,\gamma_{l-1})\otimes \tau.
    \]
    Here, we use  the naturality of $f$ at $p^2_0$. 
    As $(\alpha_0,\alpha_1,\dots ,\alpha_{m-1})\otimes \sigma$ is a minimal expression, we get $l\geq m$, so $l=m$ and the first half of this proposition is true. Using \cref{PropositionStrongUniquenessMinimalExpression},
    there exists a bijection $h\colon [m]\to [m]$ such that $h^{-1} \cdot \tau = \sigma$ and $\gamma_{h(i)}=\alpha_i$ for all $i<m$. 
    Therefore, we obtain
    \[
    \underline{f}\left(
    \begin{pmatrix}
        a_0\\
        b_0
    \end{pmatrix},
    \dots,
    \begin{pmatrix}
        a_{n-1}\\
        b_{n-1}
    \end{pmatrix}
    \right)=
    \left(
    \begin{pmatrix}
        \gamma_0\\
        \delta_0
    \end{pmatrix},
    \dots,
    \begin{pmatrix}
        \gamma_{m-1}\\
        \delta_{m-1}
    \end{pmatrix}
    \right)
    \otimes (h \cdot\sigma)
    =
    \left(
    \begin{pmatrix}
        \alpha_0\\
        \delta_{h(0)}
    \end{pmatrix},
    \dots,
    \begin{pmatrix}
        \alpha_{m-1}\\
        \delta_{h(m-1)}
    \end{pmatrix}
    \right)
    \otimes 
    \sigma
    \]
    and this equation completes the proof of the existence. 
    The uniqueness of $\beta_i$ follows from \cref{PropositionStrongUniquenessMinimalExpression}, since $\alpha_{0}, \dots ,\alpha_{m-1}$ are mutually distinct by \cref{LemmaMinimalInjectivity}.
\end{proof}

\begin{setting}\label{SettingBeta}
For later references, we summarize a list of settings:
    \begin{itemize}
        \item All of \cref{SettingZero} $\subset$ \cref{SettingOne} $\subset$ \cref{SettingAlpha}.
        \item Let $\phi_i \colon X^n \to X\ (0\leq i<m)$ denote the function that sends an $n$-tuple of elements $(b_0, \dots ,b_{n-1})$ to $\beta_{i}$ defined in \cref{PropositionInducedFunction}.
    \end{itemize}
\end{setting}

\begin{lemma}\label{lemmaPhiCommutesp}\memo{\cref{PropositionPhiPreservingLstr} wo simesunoni irunode tuika.}
For any $(b_0,b_1,\cdots, b_{n-1})\in X^n$ and $i,j,l\in \N$ which satisfies $j<l$ and $i<n$, we have 
\[
\phi_i\bigl(p^l_j(b_0),p^l_j(b_1),\cdots, p^l_j(b_{n-1})\bigr)=p^l_j\bigl(\phi_i(b_0,b_1,\cdots, b_{n-1})\bigr).
\]
\end{lemma}
\begin{proof}
    By definition of $\phi_i$, we have
    \[
        \underline{f}\left(
    \begin{pmatrix}
        a_0\\
        b_0
    \end{pmatrix},
    \dots,
    \begin{pmatrix}
        a_{n-1}\\
        b_{n-1}
    \end{pmatrix}
    \right) = 
    \left(
    \begin{pmatrix}
        \alpha_0\\
        \phi_0(b)
    \end{pmatrix},
    \dots,
    \begin{pmatrix}
        \alpha_{m-1}\\
        \phi_{m-1}(b)
    \end{pmatrix}
    \right)
    \otimes \sigma,
    \]
    where $b$ denotes the $n$-tuple $(b_0,b_1,\cdots, b_{n-1})$.
    Applying the edge action of $p'^l_j$, we obtain
    \[
        \underline{f}\left(
    \begin{pmatrix}
        a_0\\
        p^l_j(b_0)
    \end{pmatrix},
    \dots,
    \begin{pmatrix}
        a_{n-1}\\
        p^l_j(b_{n-1})
    \end{pmatrix}
    \right) = 
    \left(
    \begin{pmatrix}
        \alpha_0\\
        p^l_j(\phi_0(b))
    \end{pmatrix},
    \dots,
    \begin{pmatrix}
        \alpha_{m-1}\\
        p^l_j(\phi_{m-1}(b))
    \end{pmatrix}
    \right)
    \otimes \sigma.
    \]
    Here, we use the fact that $\underline{f}$ commutes with the edge action of $p'^l_j$. This equation completes the proof. 
\end{proof}

\begin{proposition}\label{PropositionCommutesOfSection}
    Under the \cref{SettingBeta}, we have the commutative diagram
    \[
    \begin{tikzcd}
    &X^m\ar[d,"{-}\otimes \sigma"]\\
        X^n\ar[r, "\underline{f}"']\ar[ru, "\prod_{0\leq i< m}\phi_i"]& \underline{\X \otimes F}
    \end{tikzcd}
    \]
    in the category of sets $\Set$, not (yet) in $\E_{\L}$.
\end{proposition}
\begin{proof}
    This proposition is obtained by applying $p^2_1\otimes F$ to the equation in \cref{PropositionInducedFunction}.
\end{proof}

\begin{proposition}\label{PropositionPhiPreservingLstr}
    Suppose \cref{SettingBeta}, and assume $\lm\in \Lm$ and $b_0,b_1,\dots, b_{n-1} \in \Image(s_\lm)\subset X$. Then, $\phi_i(b_0,\dots ,b_{n-1})\in \Image{(s_\lm)}$ for all $i<m$. In addition, $\phi_i$ is a morphism of $\Lstr$ for all $i<m$. 
\end{proposition}
\begin{proof}
    Take $\lm$ and $b_0,b_1,\dots ,b_{n-1}$ as above. By the definition of $s'_\lm$, we can take $c_i \in \X(v_{\lm}) 
    $ which satisfies $s'_\lm(c_i)=    \begin{pmatrix}
        a_i\\
        b_i
    \end{pmatrix}$. We take $l$ and $\gamma_0,\gamma_1,\dots,\gamma_{l-1}\in \X(v_{\lm})$ so that
    \[
    (f_{v_{\lm}})(c_0,c_1,\dots, c_{n-1})=(\gamma_0,\gamma_1,\dots ,\gamma_{l-1}) \otimes (\tau\in F[l]).
    \]
    Let $b=(b_0,b_1,\dots ,b_{n-1}) \in X^n$. By the naturality of $f$ at $s_{\lm}'$, we obtain the the following commutative diagram    
\[
\begin{tikzcd}
{\mathbb{X}(v_\lambda)^n = (X \times Q_\lambda)^n}
  & {\underline{\mathbb{X}}^n = X^n} \\
{\mathbb{X}(v_\lambda) \otimes F}
  & {\underline{\mathbb{X}} \otimes F},
\arrow[from=1-1, to=1-2, "({s'_\lambda})^n"]
\arrow[from=1-1, to=2-1, "f_{v_\lambda}", swap]
\arrow[from=1-2, to=2-2, "\underline{f}"]
\arrow[from=2-1, to=2-2, "s'_\lambda \otimes F", swap]
\end{tikzcd}
\]
which implies the equation
\[
\begin{tikzcd}[
]
{(c_0, \dots, c_{n-1})}
  & {\left(
    \begin{pmatrix}
        a_0\\
        b_0
    \end{pmatrix},
    \dots,
    \begin{pmatrix}
        a_{n-1}\\
        b_{n-1}
    \end{pmatrix}
    \right)} \\
{(\gamma_0, \dots, \gamma_{\ell-1}) \otimes \tau}
  & {(s'_{\lambda}(\gamma_0),\dots, s'_{\lambda}(\gamma_{\ell-1}))\otimes \tau=\left(
    \begin{pmatrix}
        \alpha_0\\
        \phi_0(b)
    \end{pmatrix},
    \dots,
    \begin{pmatrix}
        \alpha_{m-1}\\
        \phi_{m-1}(b)
    \end{pmatrix}
    \right)
    \otimes \sigma.}
\arrow[from=1-1, to=1-2, mapsto]
\arrow[from=2-1, to=2-2, mapsto]
\arrow[from=1-1, to=2-1, mapsto]
\arrow[from=1-2, to=2-2, mapsto]
\end{tikzcd}
\]
    By \cref{LemmaStrongUniqunessMinimalExpression}, this equality implies $\Image{(s'_\lm)}\supset \left\{    \begin{pmatrix}
        \alpha_0\\
        \phi_0(b)
    \end{pmatrix},
    \dots,
    \begin{pmatrix}
        \alpha_{m-1}\\
        \phi_{m-1}(b)
    \end{pmatrix}\right\}$. From the definition of $s'_{\lm}$, it follows that for any $i < m$ there exists $\gamma_i\in X$, $\delta_i \in Q_{\lm}^X$ which satisfies
    $\begin{pmatrix}
        \gamma_i\\
        s_{\lm}(\delta_i)
    \end{pmatrix}
    =
    \begin{pmatrix}
        \alpha_i\\
        \phi_i(b)
    \end{pmatrix}
    $. This shows $\phi_{i}(b) \in \Image(s_\lm)\subset X$. 

    Consider the diagram below. (To avoid confusion, we write $p_{i,n_{\lm}}$ and $Q_{\lm,X}$ instead of $p_{i}^{n_{\lm}}$ and $Q_{\lm}^{X}$; All superscripts in the diagram mean exponents.) By \cref{lemmaPhiCommutesp}, the upper square commutes. Due to the first half of the statements, we obtain $\phi_i(\Image(s_{\lm}^n))\subset\Image(s_{\lm})$. 
    Therefore, we obtain $\phi_i^{n_{\lm}}\bigl(Q_{\lm,X^n}\bigr)\subset Q_{\lm,X}$.  
    
    \begin{center}
    \begin{tikzcd}[row sep=2.5em, column sep=10em]
	X^n & X \\
	(X^n)^{n_\lambda} & X^{n_\lambda} \\
	Q_{\lambda,X^n} & Q_{\lambda,X}
	\arrow[from=1-1, to=2-1, "{\prod_{i\in [n_{\lambda}]} (p_{i,n_{\lambda}})^n }","\cong"']
	\arrow[from=1-2, to=2-2, "{\prod_{i\in [n_{\lambda}]} p_{i,n_{\lambda}} }"',"\cong"]
	\arrow[from=2-1, to=2-2, "\phi_i^{n_\lambda}"]
	\arrow[from=3-1, to=2-1, rightarrowtail]
	\arrow[from=3-2, to=2-2, rightarrowtail]
	\arrow[from=3-1, to=3-2, dashed]
	\arrow[from=1-1, to=1-2, "\phi_i"]
	\arrow[from=3-1, to=1-1, bend left=60,"s_{\lm}^n"]
	\arrow[from=3-2, to=1-2, bend right=60,"s_{\lm}"']
    \end{tikzcd}
    \end{center}

    
\end{proof}

\begin{setting}\label{SettingRigid}
For the later references, we summarize a list of settings:
\begin{itemize}
    \item All of \cref{SettingZero} $\subset$ \cref{SettingOne} $\subset$ \cref{SettingAlpha} $\subset$ \cref{SettingBeta}.
    \item $(X,\{Q_{\lm}^X\}_{\lm\in\Lm})$ is an inhabited-lex-rigid object of $\Lstr$. 
\end{itemize}
\end{setting}

\begin{proposition}\label{PropositionAllMorphismAreNaive}
        Under the \cref{SettingRigid}, there exists $h\colon m\to n$ which makes the following diagram commutative in the encoding topos $\E_{\L}$:
    \[
    \begin{tikzcd}
    &\X^m\ar[d,"{-}\otimes \sigma"]\\
        \X^n\ar[r, "f"']\ar[ru, "{\X}^h"]& \X \otimes F
    \end{tikzcd}
    \]
    In other words, we have $f={-}\otimes \tau$, where $\tau= h\cdot \sigma$. In addition, such $\tau \in F[n]$ is unique. 
\end{proposition}
\begin{proof}
    As $n>0$ and the $\L$-structure $(X, \{Q_\lm^X \subset X^{n_\lm}\}_{\lm \in \Lm})$ is inhabited-lex-rigid, \cref{PropositionPhiPreservingLstr} implies that $\phi_i$ is a projection map for all $i<m$. 
    So, we define $h\colon [m]\to [n]$ as the unique function which satisfies $\phi_i=p_{h(i)}$ for all $i<m$. (Do not confuse the $i$-th projection $p_i:X^n\to X$ with $p^n_i:X\to X$.) We will show this $h$ satisfies the condition of the statements. 
    
 The commutativity at the main vertex $v$ follows from \cref{PropositionCommutesOfSection}. 
 The commutativity at the vertex
 $v_{\lm}$ follows from the fact that $s'_{\lm}$-action is a section of $r_{\lm}$-action.

 We show the uniqueness of $\tau$. Assume ${-}\otimes \tau={-}\otimes \tau'$. As $\#X\geq \aleph_0$, we can take pairwise distinct elements $x_0,x_1,\dots ,x_{n-1}\in X$. Substituting $(x_0,x_1,\dots ,x_{n-1})$ in this equation, we obtain $(x_0,x_1,\dots, x_{n-1}) \otimes \tau=(x_0,x_1,\dots ,x_{n-1})\otimes \tau'$. By \cref{LemmaCancellationInjection}, $\tau =\tau'$. 
\end{proof}

\begin{remark}
    We write down the hypothesis of the above proposition explicitly.
    \begin{itemize}
        \item $\L=\{Q_{\lm}\}_{\lm \in \Lm}$ is a relational language and $\L$ has no nullary relational symbol.
	\item $(X,\{Q_\lm^X \subset X^{n_\lm}\}_{\lm \in \Lm})$ is an inhabited-lex-rigid $\L$-structure.
        \item $(\#X)^{\aleph_0}=\#X \geq \aleph_0$ and $Q_\lm$ is non-empty for every $\lm \in \Lm$.
        \item $\X \in \ob(\E_{\L})$ is a presheaf encoding of the $\L$-structure $(X,\{Q_\lm^X \subset X^{n_\lm}\}_{\lm \in \Lm} )$.
	\item $F$ is an object of the topos $\OC$, and assume $F$ is in the essential image of $\iota_{\ast}\colon \iOC\to\OC$.
        \item $f\colon \X^n \to \X \otimes F$ is a morphism in the encoding topos $\E_{\L}$, where $n$ is a \rm{positive} integer.
        \item Let $m$ denote the maximum length of $f(\underline{\X^n})$. \Cref{ProppositionImagefBounded} guarantees that such $m$ exists.
    \end{itemize}
    
\end{remark}

\begin{theorem}\label{MainTheorem}
    For some countable relational language $\L$, the presheaf topos $\E_{\L}$ has proper class many (mutually non-isomorphic) inhabited-topos-rigid objects. As a consequence, $\E_{\L}$ has proper class many (mutually non-equivalent) quotient topoi.
\end{theorem}
\begin{proof}
    Assume the relational language $\L=\{Q_{\lm}\}_{\lm\in\Lm}$ satisfies condition of \cref{PropositionManyNonEmptyLexRigid}. We show this $\L$ satisfies the condition of the theorem. Take an arbitrary infinite cardinal $\kappa$, and let $\mathcal{X}=(X,Q_{\lm}^X)$ be an $\L$-structure that satisfies condition of \cref{PropositionManyNonEmptyLexRigid}. As $(\#X)^{\aleph_0}=\#X$ and $Q_{\lm}^X\neq \emptyset$ for each $\lm \in \Lm$, there exists a presheaf encoding of the $\L$-structure $\mathcal{X}$, $\X\in\ob(\E_{\L})$. This is an inhabited-topos-rigid object by \cref{PropositionAllMorphismAreNaive} and \cref{PropositionUnwindingInhabitedToposRigidity}. 


    As we can take $\kappa$ arbitrary large, the cardinality of $\underline{\X}$ is unbounded for inhabited-topos-rigid $\X$. So, there exist (mutually non-isomorphic) proper class many inhabited-topos-rigid objects. The last part of this theorem follows from \cref{PropositionIsotopoiImpliesIsoobjects}. \memo{Yosasou~~}
\end{proof}

\begin{corollary}\label{CorollaryCountableMonoid}
    Let $M_{\omega}$ be the free monoid generated by a countably infinite set. Then the $M_{\omega}$-action topos $\PSh(M_{\omega})$ has proper class many (mutually non-isomorphic) quotient topoi.
\end{corollary}
\begin{proof}
Let $\L$ be a countable relational language in \cref{MainTheorem}.
    Consider a monoid $N_{\L}$, freely generated by the set $\{p^n_i\colon v \to v\}_{0\leq i <n\leq\omega}\cup\{p'^n_i\colon v \to v\}_{0\leq i <n<\omega}\cup \{e_{\lm}\}_{\lm \in \Lm}$, and the relations $\{e_{\lm}e_{\lm} = e_{\lm}\}_{\lm \in \Lm}$. Considering the Cauchy completion of the monoid $N_{\L}$, the encoding topos $\E_{\L}$ is equivalent to the presheaf topos $\PSh(N_{\L})$. Since $N$ is a countable monoid, there exists a surjective monoid homomorphism $M_{\omega}\twoheadrightarrow N_{\L}$, which induces a connected geometric morphism $\PSh(M_{\omega}) \to \PSh(N_{\L}) \simeq \E_{\L}$. 
    In fact, if a functor $\C \to \D$ between small categories $\C$ and $\D$ is full and bijective on objects, then the induced geometric morphism $\PSh(\C)\to \PSh(\D)$ is hyperconnected \cite[A.4.6.9]{johnstone2002sketchesv1} and therefore connected \cite[C.1.5]{johnstone2002sketchesv2}.
    This completes the proof since a composite of two connected geometric morphisms is again connected.
\end{proof}

\begin{remark}
In fact, \cref{MainTheorem} and \cref{CorollaryCountableMonoid} are equivalent.
Let $\L$ be a countable relational language.
    Forcing $e_{\lm}$ to be the neutral element, we can construct a surjective monoid homomorphism in the opposite way $N_{\L} \twoheadrightarrow M_{\omega}$ as well. 
    The two topoi $\PSh(M_{\omega})$ and $\E_{\L}$ are quotient topoi of each other.
\end{remark}
\bibliographystyle{alpha}
\bibliography{BiblioKamioLawvere} 

\begin{thebibliography}{MLM94}

\bibitem[AJ21]{anel2021topo}
Mathieu Anel and Andr{\'e} Joyal.
\newblock Topo-logie.
\newblock {\em New Spaces in Mathematics: Formal and Conceptual Reflections}, 1:155--257, 2021.

\bibitem[AR94]{adamek1994locally}
Jiri Adamek and Jiri Rosicky.
\newblock {\em Locally Presentable and Accessible Categories}.
\newblock Cambridge University Press, 1994.

\bibitem[Bek04]{beke2004theories}
Tibor Beke.
\newblock Theories of presheaf type.
\newblock {\em The Journal of Symbolic Logic}, 69(3):923--934, 2004.

\bibitem[Bor94]{borceux1994handbookv3}
Francis Borceux.
\newblock {\em Handbook of Categorical Algebra: Volume 3, Sheaf Theory}.
\newblock Cambridge University Press, 1994.

\bibitem[Car18]{caramello2018theories}
Olivia Caramello.
\newblock {\em Theories, Sites, Toposes: Relating and studying mathematical theories through topos-theoretic `bridges'}.
\newblock Oxford University Press, 2018.

\bibitem[EBV02]{el2002simultaneously}
Robert El~Bashir and Jiri Velebil.
\newblock Simultaneously reflective and coreflective subcategories of presheaves.
\newblock {\em Theory and Applications of Categories}, 10(16):410--423, 2002.

\bibitem[Hen18]{henry2018localic}
Simon Henry.
\newblock The localic isotropy group of a topos.
\newblock {\em Theory and Applications of Categories}, 33(41):1318--1345, 2018.

\bibitem[HK24]{hora2024quotient}
Ryuya Hora and Yuhi Kamio.
\newblock Quotient toposes of discrete dynamical systems.
\newblock {\em Journal of Pure and Applied Algebra}, 228(8):107657, 2024.

\bibitem[Hor24]{hora2024internal}
Ryuya Hora.
\newblock Internal parameterization of hyperconnected quotients.
\newblock {\em Theory and Applications of Categories}, 42(11):263--313, 2024.

\bibitem[JM89]{johnstone1989local}
Peter~T Johnstone and Ieke Moerdijk.
\newblock Local maps of toposes.
\newblock {\em Proceedings of the London Mathematical Society}, 3(2):281--281, 1989.

\bibitem[Joh81]{johnstone1981factorization}
Peter~T Johnstone.
\newblock Factorization theorems for geometric morphisms, i.
\newblock {\em Cahiers de topologie et g{\'e}om{\'e}trie diff{\'e}rentielle}, 22(1):3--17, 1981.

\bibitem[Joh02a]{johnstone2002sketchesv1}
Peter~T Johnstone.
\newblock {\em Sketches of an Elephant: A Topos Theory Compendium, Volume 1}.
\newblock Oxford University Press, 2002.

\bibitem[Joh02b]{johnstone2002sketchesv2}
Peter~T Johnstone.
\newblock {\em Sketches of an Elephant: A Topos Theory Compendium: Volume 2}.
\newblock Oxford University Press, 2002.

\bibitem[Law07]{lawvere2007axiomatic}
F~William Lawvere.
\newblock Axiomatic cohesion.
\newblock {\em Theory and Applications of Categories [electronic only]}, 19:41--49, 2007.

\bibitem[Law09]{Open240411LawvereRevised}
F~William Lawvere.
\newblock open problems in topos theory.
\newblock One can find this article via \url{https://lawverearchives.com/posthumous-publications/}, April 2009.

\bibitem[Lur09]{lurie2009higher}
Jacob Lurie.
\newblock {\em Higher topos theory}.
\newblock Princeton University Press, 2009.

\bibitem[ML98]{Maclane1998CWMcategories}
Saunders Mac~Lane.
\newblock {\em Categories for the working mathematician}.
\newblock Graduate Texts in Mathematics. Springer-Verlag, New York, second edition, 1998.

\bibitem[MLM94]{maclane1994sheaves}
Saunders Mac~Lane and Ieke Moerdijk.
\newblock {\em Sheaves in geometry and logic: A first introduction to topos theory}.
\newblock Springer Science \& Business Media, 1994.

\bibitem[Ros82]{rosenthal1982quotient}
Kimmo~I Rosenthal.
\newblock Quotient systems in grothendieck topoi.
\newblock {\em Cahiers de topologie et g{\'e}om{\'e}trie diff{\'e}rentielle}, 23(4):425--438, 1982.

\bibitem[VPH65]{vopvenka1965rigid}
Petr Vop{\v{e}}nka, Ale{\v{s}} Pultr, and Zden{\v{e}}k Hedrl{\'\i}n.
\newblock A rigid relation exists on any set.
\newblock {\em Commentationes Mathematicae Universitatis Carolinae}, 6(2):149--155, 1965.

\end{thebibliography}

\end{document}